\documentclass{amsart}

\usepackage{amsmath}
\usepackage{amsfonts}
\usepackage{url}

\usepackage[hcentering,margin=1.1in]{geometry}

\usepackage{amsthm}
\newtheorem{theorem}{Theorem}

\theoremstyle{definition}

\newcommand\R{\mathbb{R}}

\newcommand\rnk[1]{n}

\newcommand\sphharmd[2]{P_{#1;#2}}
\newcommand\inlinefrac[2]{#1/#2}
\newcommand\inprod[2]{\langle#1,#2\rangle}
\newcommand\fixed[1]{\dot{#1}}
\newcommand\xfixed{\fixed{x}}
\newcommand\latshell[2]{\Lambda_{#1}({#2})}
\newcommand\genlatshell[2]{\mathcal{L}_{#1}(#2)}

\newcommand\inprodcountlp[2]{M_{#1}(L;#2)}
\newcommand\minLoz[1]{m_0}
\DeclareMathOperator\minnorm{min}

\newcommand\term[1]{\textit{#1}}

\newcommand\thmozredux{2}

\usepackage{verbatim}

\title[Refined Configuration Results for Extremal {Type~II}
Lattices]{Refined Configuration Results for\\Extremal {Type~II}
Lattices of Ranks $40$ and $80$}

\author[N.\ D.\ Elkies]{Noam D.\ Elkies}
\address{Department of Mathematics, Harvard University\newline\indent
One Oxford Street\newline \indent Cambridge, MA 02138}
\email{elkies@math.harvard.edu}

\author[S.\ D.\ Kominers]{Scott D.\ Kominers}
\address{Department of Mathematics, Harvard University\newline\indent c/o
8520 Burning Tree Road\newline \indent Bethesda, MD 20817}
\email{kominers@fas.harvard.edu, skominers@gmail.com}


\begin{document}
\begin{abstract}
 We show that, if $L$ is an extremal
 Type~II lattice of rank $40$ or $80$,
 then $L$ is generated by its vectors of norm $\minnorm(L)+2$.
 This sharpens earlier results of Ozeki, and the second author
 and Abel, which showed that such lattices $L$ are generated by
 their vectors of norms $\minnorm(L)$ and $\minnorm(L)+2$.
\end{abstract}
\maketitle

\section{Introduction}\label{intro}
\subsection{Preliminaries}
A \term{lattice} of \term{rank}~$n$ is a free $\mathbb{Z}$-module
of rank $n$ equipped with a positive-definite inner product
$\inprod{\cdot}{\cdot}: L\times L\to\mathbb{R}$.  For $x\in L$, we
call $\inprod{x}{x}$ the \term{norm}\/ of $x$. If
$\inprod{x}{x'}\in\mathbb{Z}$ for all $x,x'\in L$, then we
say that $L$ is \term{integral}.   The lattice $L$ is called
\term{even}\/ if every vector of $L$ has an even integer norm,
i.e.~if $\inprod{x}{x}\in 2\mathbb{Z}$ for all $x\in L$.
The well-known parallelogram identity,
\begin{equation}\label{eq:paridentity}2\inprod{x}{x'}=
\inprod{x+x'}{x+x'}-\inprod{x}{x}-\inprod{x'}{x'},\end{equation}
shows that an even lattice must be integral.

The \term{dual} of~$L$, denoted $L^*$, is the lattice defined by
$L^* = \left\{x'\in L\otimes\mathbb{R}: \inprod{x}{x'}\in\mathbb{Z}
\text{ for all $x\in L$} \right\}$.  If~$L=L^*$ then~$L$ is said to
be \term{self-dual}. 	A self-dual lattice is said to be of
\term{Type~II}\/ if it is even, and of \term{Type~I}\/ otherwise.

The rank of a Type~II lattice must be divisible by~$8$
(see~\cite[p.~53 (Cor.~2)]{Serre:course} and
\cite[p.~109]{Serre:course}). Mallows, Odlyzko, and
Sloane~\cite{MallowsOdlyzkoSloane} (see also \cite[p.~194]{SPLAG})
used theta functions and modular forms to show that the
\term{minimal nonzero norm}\/ (or just \term{minimal norm})
$\minnorm(L)$ of vectors in a Type~II lattice~$L$ of rank $n$
satisfies the upper bound
$ \minnorm(L) \leq 2\lfloor\inlinefrac{n}{24}\rfloor + 2 $.
The lattice $L$ is called \term{extremal}\/ if it attains equality
in this bound, that is, if
$$
 \minnorm(L) = 2\lfloor\inlinefrac{n}{24}\rfloor + 2.
$$

\subsection{Configuration Results for Extremal Type~II Lattices}
Venkov~\cite{Venkov:32} and Ozeki~\cite{Ozeki:32} showed that
if~$L$ is an extremal Type~II lattice of rank $32$ then
$L$ is generated by its vectors of minimal norm.  Additionally,
Ozeki~\cite{Ozeki:48} proved an analogous result for extremal Type~II lattices
of rank~$48$.  Recently, the second author~\cite{Kominers:56+72+96}
extended the methods of Ozeki~\cite{Ozeki:48} to obtain
analogous results for
extremal Type~II lattices of ranks~$56$, $72$, and $96$.

The following similar but weaker result has been obtained for extremal Type~II
lattices of ranks $40$, $80$, and $120$.

\begin{theorem}\label{thm:oz40r}
If $L$ is an extremal Type~II lattice of rank $40$, $80$, or $120$, then
$L$ is generated by its vectors of norms $\minnorm(L)$ and $\minnorm(L)+2$.
\end{theorem}

The rank~$40$ case of this result is due to
Ozeki~\cite{Ozeki:40}, while the cases of ranks $80$ and~$120$
are due to the second author and Abel~\cite{Kominers+Abel:40r}.
Ozeki~\cite{Ozeki:40} noted that
the rank~$40$ case of Theorem~\ref{thm:oz40r} is sharp,
in the sense that there exist extremal Type~II lattices of rank $40$
that are not generated by their vectors of minimal norm.

In this note, we prove the following refinement of Theorem~\ref{thm:oz40r}.
\begin{theorem}\label{thm:oz40/redux}
If $L$ is an extremal Type~II lattice of rank $40$ or $80$, then
$L$ is generated by its vectors of norm $\minnorm(L)+2$.
\end{theorem}

\section{The Proof of Theorem~\thmozredux}

Let~$L$ be an extremal Type~II lattice of rank $n$, and let
$\minLoz{n}=\minnorm(L)$.
Adopting the notation of~\cite{Kominers+Abel:40r}, we write
$\latshell{j}{L}= \{x\in L: \inprod{x}{x}=j\}$
for the set of norm-$j$ vectors of $L$, and denote
by $\genlatshell{j}{L}$ the lattice generated by $\latshell{j}{L}$.
Also following~\cite{Kominers+Abel:40r}, we define, for $\xfixed \in \R^n$,
$$
\inprodcountlp{j}{\xfixed} =
  \left\vert
    \{x\in \latshell{\minLoz{n}+2}{L}:
    \inprod{\xfixed}{x}=j\}
   \right\vert,
$$
noting that $\inprodcountlp{-j}{\xfixed}=\inprodcountlp{j}{\xfixed}$
for any $\xfixed$ and $j$.

We recall that Venkov~\cite{Venkov:reseaux} proved using
weighted theta functions and modular forms that, for each $j>0$
such that $\latshell{j}{L} \neq \emptyset$, the shell
$\latshell{j}{L}$ is a spherical $(t\frac{1}{2})$-design, where
$t$ is $11$, $7$, or $3$ according as~$n$ is
congruent to~$0$, $8$, or $16$ modulo~$24$.
This means that
\begin{equation}
\label{eq:vanishsum}
  \sum_{x\in \latshell{j}{L}} P(x) = 0
\end{equation}
holds for all spherical harmonic polynomials $P$\/ of degree~$s$ with
$1 \leq s \leq t$ as well as $s = t+3$.  We shall use \eqref{eq:vanishsum}
with $P(x) = \sphharmd{s}{\xfixed}(x)$ for $s$ even and $\xfixed\in \R^n$,
where $\sphharmd{d}{\xfixed}$ is the zonal spherical harmonic
polynomial of degree~$d$ (see~\cite{Vilenkin}).

Now take $n = 40$ or $n = 80$, so that $m_0 = n/10$ and $t = m_0 - 1$.
Using the modularity of the theta function of~$L$, we calculate that
$|\latshell{m_0+2}{L}|$ is $87859200$ for $n=40$, and $7541401190400$
for $n=80$; in~particular $\latshell{m_0+2}{L} \neq \emptyset$.
Let $S_n$ be the set of $m_0/2$ even values of~$s$ for which
we can apply \eqref{eq:vanishsum} to $\sphharmd{s}{\xfixed}$;
that is, $S_{40} = \{2,6\}$ and $S_{80} = \{2,4,6,10\}$.

We now proceed to prove Theorem~\ref{thm:oz40/redux}.

\begin{proof}[Proof of Theorem~\ref{thm:oz40/redux}]
By Theorem~\ref{thm:oz40r}, it suffices to show that every vector in
$\latshell{\minLoz{n}}{L}$ is contained in
$\genlatshell{\minLoz{n}+2}{L}$. We therefore suppose that there is
some $\xfixed\in\latshell{\minLoz{n}}{L}$ not in
$\genlatshell{\minLoz{n}+2}{L}$, seeking a contradiction.

For any $x\in \latshell{\minLoz{n}+2}{L}$, we must have
\begin{equation}\label{eq:inprodbound}
\vert \inprod{x}{\xfixed} \vert  \in  \left\{0,1,\ldots,
\frac{\minLoz{n}}{2}-1,\frac{\minLoz{n}}{2}+1\right\}.
\end{equation}  Indeed, if $\inprod{\xfixed}{\pm
x}>\frac{\minLoz{n}}{2}+1$, then $[\xfixed]$ contains a nonzero vector
$x\mp \xfixed$ of norm $$\inprod{x\mp \xfixed}{x\mp
\xfixed}=\inprod{x}{x}\mp
2\inprod{x}{\xfixed}+\inprod{\xfixed}{\xfixed}<\inprod{\xfixed}{\xfixed}=\minLoz{n}$$
by~\eqref{eq:paridentity}, contradicting the extremality of $L$.
Furthermore, we may assume
$\inprod{x}{\xfixed} \neq \pm \inlinefrac{\minLoz{n}}{2}$,
else we would have $x\mp \xfixed\in
\latshell{\minLoz{n}+2}{L}$ by~\eqref{eq:paridentity}, whence
$x = \pm\xfixed+(x\mp \xfixed)\in \genlatshell{\minLoz{n}+2}{L}$
follows.  That is, we may assume $\inprodcountlp{m_0/2}{\xfixed} = 0$.

Now, when we take $P = \sphharmd{s}{\xfixed}$
in~\eqref{eq:vanishsum}, we obtain $P(x) = Q(\inprod{x}{\xfixed})$
for some even polynomial~$Q$\/ in one variable.
For any such polynomial~$Q$, we have
\begin{align*}
\sum_{x\in\latshell{\minLoz{n}+2}{L}} Q(\inprod{x}{\xfixed})
 & = Q(0) \cdot \inprodcountlp{0}{\xfixed}
   + 2\sum_{j=1}^{(\inlinefrac{\minLoz{n}}{2})+1}
     Q(j) \cdot\inprodcountlp{j}{\xfixed}
   \\
 & = Q(0) \cdot \inprodcountlp{0}{\xfixed}
   + 2\sum_{j=1}^{(\inlinefrac{\minLoz{n}}{2})-1}
     Q(j) \cdot\inprodcountlp{j}{\xfixed}
   + 2 Q\left(\frac{m_0}{2}+1 \right)
     \cdot\inprodcountlp{(\inlinefrac{\minLoz{n}}{2})+1}{\xfixed}.
\end{align*}
Letting $s$ vary over $S_n$, we obtain
$\inlinefrac{\minLoz{n}}{2}$ homogeneous linear equations in the
$(\inlinefrac{\minLoz{n}}{2})+1$ variables
$$\inprodcountlp{0}{\xfixed},\ldots,\inprodcountlp{(\inlinefrac{\minLoz{n}}{2})-1}{\xfixed},
\inprodcountlp{(\inlinefrac{\minLoz{n}}{2})+1}{\xfixed}.$$
In each of our cases $n=40$ and $n=80$, we find that these equations
are linearly independent and thus have a unique solution up to scaling.
We compute that for $n=40$ the $\inprodcountlp{j}{\xfixed}$ for
$j=0,1,3$ are proportional to
$$
4688, \, 4293, \, -37,
$$
while for $n=80$ the $\inprodcountlp{j}{\xfixed}$ for
$j=0,1,2,3,5$ are proportional to
$$
5661456, \, 3946750, \, 711000, \, 88875, \, -553.
$$
But this is impossible because
$\inprodcountlp{j}{\xfixed}\geq 0$ for all~$j$
by definition,
and if every $\inprodcountlp{j}{\xfixed}$ vanished
then we would have $\latshell{\minLoz{n}+2}{L} = \emptyset$.
\end{proof}

\section{Remarks}

It is not known whether the conclusion of Theorem~\ref{thm:oz40/redux}
holds for $L$ extremal and Type~II of rank~$120$.  We cannot directly
adapt our approach to this case, because when $n=120$
(and $\minLoz{n}=12$) the system used in
the proof of Theorem~\ref{thm:oz40/redux} yields nonnegative ratios
among the values $\inprodcountlp{j}{\xfixed}$ for
$j \in \{0,1,2,3,4,5,7\}$.
However, the first author~\cite{Elkies:40r} has obtained a different
refinement of
Theorem~\ref{thm:oz40r} in the rank-$120$ case: any extremal Type~II
lattice of rank~$120$ is generated by its vectors of minimal norm.

\section*{Acknowledgements}
The first author was partly supported by National Science Foundation
grant DMS-0501029.  The second author was partly supported by a
Harvard Mathematics Department Highbridge Fellowship and a Harvard
College Program for Research in Science and Engineering Fellowship.

\bibliographystyle{amsalpha}
\bibliography{references}

\end{document}